\documentclass[11pt,eqright]{amsart}
\usepackage{amssymb,amsmath}
\usepackage{bm}
\usepackage{amsfonts}
\usepackage{epsf}
\usepackage{graphicx}
\newtheorem{theorem}{Theorem}[section]
\newtheorem{lemma}[theorem]{Lemma}
\newtheorem{corollary}[theorem]{Corollary}

\newtheorem{remark}[theorem]{Remark}

\theoremstyle{definition}
\usepackage[titletoc]{appendix}
\usepackage{subfigure}
\numberwithin{equation}{section}
\numberwithin{table}{section}

\newcommand{\bc}{\begin{center}}
\newcommand{\ec}{\end{center}}
\newcommand{\be}{\begin{eqnarray}}
\newcommand{\ee}{\end{eqnarray}}
\newcommand{\nn}{\nonumber}
\newcommand{\ben}{\begin{eqnarray*}}
\newcommand{\een}{\end{eqnarray*}}

\newcommand{\Om}{\Omega}

\newcommand{\lam}{\lambda}

\newcommand{\na}{\nabla}

\def\x{\times}

\def\na{\nabla}

\def\cE{\mathcal{E}}

\def\cT{\mathcal{T}}
\def\R{\mathbb{R}}

\def\div{\operatorname{div}}

\DeclareMathOperator{\sspan}{span}

\DeclareMathOperator{\dif}{d}

\newcommand\diff{\,\dif}

\title[]
{\small Guaranteed Lower and upper bounds for eigenvalues of second order elliptic operators in any dimension}
\author[J.~Hu]
{Jun Hu$^\ast$}
\address{$^\ast$ LMAM and School of Mathematical Sciences,
 Peking University, Beijing 100871, P. R. China}
\email{hujun@math.pku.edu.cn}
\author[R. Ma]{Rui Ma$^\dagger$}
\address{$^\dagger$ LMAM and School of Mathematical Sciences,
 Peking University, Beijing 100871, P. R. China}
\email{maruipku@gmail.com}
\thanks{The  first author was supported by  the NSFC Project 11271035 and  by  the NSFC Key Project 11031006.}

\begin{document}
\newpage
\begin{abstract}
In this paper, a new method is proposed  to produce guaranteed lower bounds for eigenvalues of general second order elliptic
 operators in any dimension.  Unlike most methods in the literature,  the proposed method only needs to solve  one discrete
  eigenvalue problem  but not involves  any base or intermediate eigenvalue problems, and does not need any
  a priori information concerning  exact eigenvalues either.  Moreover, it just assumes basic regularity of exact eigenfunctions. This  method is  defined  by  a novel generalized Crouzeix-Raviart element which is proved to yield asymptotic lower bounds for eigenvalues of general second order elliptic operators, and a simple post-processing method.
  As a byproduct, a simple and cheap method is also proposed to obtain guaranteed upper bounds for eigenvalues, which is based
   on  generalized Crouzeix-Raviart element approximate eigenfunctions, an averaging interpolation from the the  generalized Crouzeix-Raviart element space to the conforming linear  element space, and an usual  Rayleigh-Ritz procedure.
    The ingredients for the analysis consist of a crucial
     projection property of the canonical interpolation operator of the generalized Crouzeix-Raviart element,
       explicitly computable
      constants for two interpolation operators.  Numerics are provided to  demonstrate the theoretical
       results.
\end{abstract}

\maketitle
\section{Introduction}
Finding eigenvalues of partial differential operators is important in the mathematical science. Since exact eigenvalues are almost impossible, many papers and books investigate their bounds from above and below.
It is well known that upper bounds for the eigenvalues can always be found by the Rayleigh-Ritz method. While the problem of obtaining lower bounds is generally considering more difficult. The study of lower bounds for eigenvalues can date back to several remarkable works.  The finite difference method \cite{Weinberger56,Weinberger1958} can provide lower bounds on  eigenvalues of the Laplace operator  on domains of regular shape without reentrant corners. The intermediate method, developed by Weinstein \cite{Weinstein1963} admits the approximate eigenvalue from below, which, somehow, heavily
  depends on  some base problem with an explicit knowledge of eigenvalues and eigenfunctions.  Both the Kato and Lehmann-Goerisch methods can produce  lower bounds for up to the $\ell$-th eigenvalue provided that the lower bound for the $(\ell+1)$-th eigenvalue is available.  In \cite{Plum1991}, Plum developed the homotopy method based on the operator comparison theorem to bound eigenvalues, which also depends on  some base problem, i.e., the one with an explicit spectrum, which is satisfied by only simple domains. If we only consider the first
eigenvalue, we can refer to a very wonderful method proposed in
\cite{Protter60}. We also refer the interested readers to
\cite{KuttlerSigillito84} for the various numerical methods for the
eigenvalues of the Laplacian operator in two dimensions.

The finite element method can effectively approximate eigenvalues with a comprehensive analysis on error estimation, see \cite{Boffi2010,StrangFix2008}. Conforming finite element methods can provide upper bounds for eigenvalues. While, some nonconforming finite element methods can give lower bounds of eigenvalues directly when the meshsize is sufficiently small, see \cite{HuHuangLin2010,Yang2010}. In \cite{HuHuangLin2010}, Hu et al. gave a comprehensive survey of  the lower bound property of eigenvalues by nonconforming finite element methods and proposed a systematic method that can produce lower bounds for eigenvalues by using nonconforming finite element methods. The theories \cite{HuHuangLin2010} were limited to asymptotic analysis and  it is not easy to check when the meshsize is small enough in practice.  Following the theory of \cite{LarrsonThomee2008,StrangFix2008}, Liu et al. \cite{LiuOishi2013} proposed  guaranteed lower bounds for eigenvalues of the Laplace operator in the two dimensions.  The main tool therein is an explicit a priori error estimation for the conforming linear  element projection.  However, for singular eigenfunctions, it needs to compute the explicit a priori error estimation by solving an auxiliary problem. Moreover, it is difficult to  generalize  the idea therein to general second order elliptic operators. Similar  guaranteed lower bounds for eigenvalues of both Laplace and biharmonic operators in two dimensions  were given by Carstensen et al., see \cite{CarstensenGallistl2013,CarstensenGedickel2013}, through using
the nonconforming Crouzeix-Raviart and Morley elements, respectively.

 The aim of this paper is to propose new methods which are able to obtain both guaranteed lower and upper bounds for eigenvalues of general second order elliptic operators in any dimension.  The method for guaranteed lower bounds
 is derived from  asymptotic lower bounds for eigenvalues produced by
  a  generalized Crouzeix-Raviart (GCR hereafter) element  proposed herein, and a simple post-processing method.
  Unlike most methods in the literature,  this new method only  needs to solve one discrete
  eigenvalue problem  but not involves  any base or intermediate eigenvalue problems, and does not need any
  a priori information concerning  exact eigenvalues either.
  The method can be regarded as an extension to the  general second order elliptic operators in any dimension of those due to \cite{LiuOishi2013} and \cite{CarstensenGallistl2013,CarstensenGedickel2013}.  Its  novelties are as follows:
  \begin{itemize}
  \item The new method  can be used to all second order elliptic operators in any dimension while those in \cite{LiuOishi2013} and \cite{CarstensenGedickel2013} only applies for the Laplace operator in two dimensions; in addition, it has  higher accuracy than those from  \cite{LiuOishi2013} and \cite{CarstensenGedickel2013}, see comparisons in Section 7.1;
  \item  The  meshsize condition \eqref{hconstrain} below improves largely that of  \cite{CarstensenGedickel2013};
   while compared with \cite{LiuOishi2013}, the method of this paper only assumes basic regularity of exact eigenfunctions.
    \end{itemize}
    The approach for guaranteed upper bounds is based  on asymptotic  upper bounds which are obtained by a
     postprocessing method firstly  proposed in \cite{HuHuangShen2011,ShenD}, see also \cite{YangHan2014}, and a Rayleigh-Ritz procedure.  Compared with     \cite{CarstensenGedickel2013} and  \cite{LuoLin}, this new method does not need to solve an eigenvalue or source problem by  a conforming finite element method.  The ingredients for the analysis consist of a crucial
     projection property of the canonical interpolation operator of the GCR element,  explicitly computable
      constants for two interpolation operators.  Numerics are provided to  demonstrate the theoretical
       results.

The remaining paper is organized as follows. Section 2 proposes  the GCR element. Section 3 proves asymptotic lower bounds for eigenvalues. Section 4 presents the guaranteed lower bounds for eigenvalues of  general elliptic operators. Section 5 provides asymptotic upper bounds for eigenvalues. Section 6 designs guaranteed upper bounds for eigenvalues. Section 7 will give some numerical tests.
\section{Preliminaries}
In this section, we present second order elliptic boundary value and eigenvalue problems and propose a generalized Crouzeix-Raviart element for them. Throughout this paper, let $\Omega\subset\R^n$ denote a bounded domain, which, for the sake of simplicity, is supposed to be a polytope.
\subsection{Second order elliptic boundary value and eigenvalue problems}
Given $f\in L^2(\Om)$, second order elliptic boundary value problems find $u\in H^1_0(\Om)$ such that
\begin{equation}\label{ellipticproblem}
(A\na u, \na v)=(f,v)\quad\text{for any }v\in H^1_0(\Omega).
\end{equation}
Here, $A$ is a matrix-valued function on $\Omega$ and satisfies
\begin{equation*}
(q, q)\lesssim (Aq, q)\quad\text{for  any } q\in (L^2(\Omega))^n,
\end{equation*}
where $p\lesssim q$ abbreviates $p\leq Cq$ for some multiplicative  mesh-size independent constant $C>0$ which may be different at different places.  Define $$\|\nabla v\|_A:=(A\nabla v, \nabla v)^{1/2}.$$
Hence $\|\nabla\cdot\|_A$ is a norm of $H^1_0(\Omega)$.
$A(x)$ is supposed to be symmetric for all $x\in\Omega$ and each component of $A$ is piecewise Lipschitz continuous on each subdomain of domain $\Omega$.

Second order elliptic eigenvalue problems find $(\lam, u)\in\R\x H^1_0(\Om)$ such that
\begin{equation}\label{eigen}
\begin{split}
(A \nabla u, \nabla v) =\lam (u, v)\quad\text{for any }v\in H^1_0(\Omega)\text{ and }\|u\|=1.
\end{split}
\end{equation}
Problem \eqref{eigen} has a sequence of eigenvalues
\begin{equation*}
0<\lam_1\leq \lam_2\leq \lam_3\leq\cdots\nearrow +\infty,
\end{equation*}
and   corresponding eigenfunctions
\begin{equation*}
u_1, u_2, u_3, \cdots,
\end{equation*}
which can be chosen to satisfy
\begin{equation*}
( u_i, u_j)=\delta_{ij}, i, j=1, 2, \cdots.
\end{equation*}
Define
\begin{equation}\label{ell}
E_{\ell}=\sspan\{u_1,u_2, \cdots, u_{\ell}\}.
\end{equation}
Eigenvalues and eigenfunctions satisfy the following
well-known Rayleigh-Ritz principle:
\begin{equation}\label{minmax}
\lam_k=\min\limits_{\dim V_k=k, V_k\subset H^1_0(\Omega)}\max\limits_{v\in
V_k}\frac{(A\nabla v,\nabla v)}{(v,v)}=\max\limits_{u\in
E_{k}}\frac{(A\nabla u,\nabla u)}{(u,u)}.
\end{equation}
\subsection{The generalized Crouzeix-Raviart element}
Suppose that $\overline{\Omega}$ is covered exactly by shape-regular partitions $\cT$  consisting of $n$-simplices  in
$n$ dimensions.  Let  $\cE$ denote the set of  all
$n-1$ dimensional subsimplices,   and  $\cE(\Omega)$ denote the set of
all the $n-1$ dimensional interior subsimplices, and  $\cE(\partial \Omega)$ denote the set of
all the $n-1$ dimensional boundary subsimplices. Given $K\in\mathcal{T}$, $h_K$ denotes the diameter of $K$ and $h:=\max_{K\in\mathcal{T}}h_K$.
Let $|K|$ denote the measure of element $K$ and $|E|$ the measure of $n-1$ dimensional subsimplex $E$. Given $E\in \cE$, let $\nu_E$ be its unit normal vector
and $[\cdot]$ be jumps of piecewise functions over $E$, namely
$$
[v]:=v|_{K^+}-v|_{K^-}
$$
for piecewise functions $v$ and any two elements $K^+$ and $K^-$ which share the common $n-1$ dimensional subsimplex $E$. Note that
$[\cdot]$  becomes  traces of functions  on $E$ for  boundary subsimplex
$E$.

Given $K\in\mathcal{T}$ and an integer $m\geq0$, let $P_m(K)$ denote the space of polynomials of degree$\leq m$ over $K$. The simplest nonconforming finite element for Problem \eqref{ellipticproblem} is the Crouzeix-Raviart (CR hereafter ) element proposed in \cite{CrouzeixRaviart(1973)}.
The corresponding element space $V_{\rm CR}$ over $\mathcal{T}$ is defined by
\begin{equation}
 V_{\rm CR}:=\begin{array}[t]{l}\big\{v\in L^2(\Om):
 v|_{K}\in P_1(K) \text{ for each }K\in \cT,
 \int_E[v]dE=0,\\
 \text{ for all } E\in\cE(\Omega)\,,
 \text{ and }\int_E vdE=0 \text{ for all }E\in\cE(\partial\Omega)
 \big\}\,.
 \end{array}\nn
\end{equation}
Since the CR element can't be proved to produce lower bounds for eigenvalues of the Laplace operator  on general meshes when eigenfunctions are smooth, see \cite{Armentano,HuHuangShen2014}. Hu et al. \cite{HuHuangLin2010} proposed the enriched Crouzeix-Raviart (ECR hereafter) element which was proved to produce lower bounds for eigenvalues of the Laplace operator in the asymptotic sense. The corresponding shape function space is as follows
 \begin{equation*}
{\rm ECR}(K):=P_1(K)+\sspan\Big\{\sum\limits_{i=1}^nx_i^2\Big\}\quad \text{for any
}K\in\cT.
 \end{equation*}
  The ECR element space $V_{\rm ECR}$ is then defined by
\begin{equation}\label{ECR}
 V_{\rm ECR}:=\begin{array}[t]{l}\big\{v\in L^2(\Om):
 v|_{K}\in {\rm ECR}(K) \text{ for each }K\in \cT,
 \int_E[v]ds=0,\\
 \text{ for all } E\in\cE(\Omega)\,,
 \text{ and }\int_E vds=0 \text{ for all   $E\in\cE(\partial\Omega)$ }
 \big\}\,.
 \end{array}\nn
\end{equation}

However, the ECR element cannot produce lower bounds for eigenvalues of general second order elliptic operators, which motivates us to generalize the ECR element to more general cases. To this end, let $\bar{A}$ be a piecewise positive-definite  constant matrix  with respect to $\mathcal{T}$, which is an approximation of $A$.
For example, we can choose $\bar{A}|_K$ to be equal to the value of $A$ at the centroid of $K$ or the integral mean on $K$. Suppose
\begin{equation}
\bar{A}|_K=\begin{pmatrix}a_{11}&a_{12}&\cdots&a_{1n}\\
a_{21}&a_{22}&\cdots&a_{2n}\\
\cdots&\cdots&\cdots&\cdots\\
a_{n1}&a_{n2}&\cdots&a_{nn}
\end{pmatrix}.
\end{equation}
Let $\bar{B}$ denote the inverse of $\bar{A}$ as follows
\begin{equation}
\bar{B}|_K=\bar{A}^{-1}|_K=\begin{pmatrix}b_{11}&b_{12}&\cdots&b_{1n}\\
b_{21}&b_{22}&\cdots&b_{2n}\\
\cdots&\cdots&\cdots&\cdots\\
b_{n1}&b_{n2}&\cdots&b_{nn}
\end{pmatrix}.
\end{equation}
The centroid of $K$ is denoted by ${\rm mid}(K)$.  The coordinate of ${\rm mid}(K)$ is denoted by $(M_1,M_2,\cdots,M_n)$. The vertices of $K$ are denoted by $a_p=(x_{1p},x_{2p},\cdots,x_{np}),1\leq p\leq n+1$. Define
$$
H=\sum^n_{i=1}b_{ii}\sum_{p<q}(x_{ip}-x_{iq})^2+2\sum_{i<j}b_{ij}\sum_{p<q}(x_{ip}-x_{iq})(x_{jp}-x_{jq}),
$$
and
\begin{equation}\label{bubblefunction}
\phi_K = \frac{n+2}{2}-\frac{n(n+1)^2(n+2)}{2H}\left(x-{\rm mid}(K)\right)^T\bar{B}|_K\left(x-{\rm mid}(K)\right).\\
\end{equation}
For two dimensions, the constant $H$ and function $\phi_K$ are presented as follows, respectively,
$$
H=b_{11}\sum_{p<q}(x_{1p}-x_{1q})^2+b_{22}\sum_{p<q}(x_{2p}-x_{2q})^2+2b_{12}\sum_{p<q}(x_{1p}-x_{1q})(x_{2p}-x_{2q}),
$$
and
\begin{equation}\label{twoDex}
\phi_K = 2-\frac{36}{H}(b_{11}(x_1-M_1)^2+b_{22}(x_2-M_2)^2+2b_{12}(x_1-M_1)(x_2-M_2)).
\end{equation}
\begin{lemma}
\label{Lemma:caculation}
Given $K\in\mathcal{T}$, there holds that
$$\frac{1}{|K|}\int_K\phi_Kdx=1.
$$
Moreover, for any $n-1$ dimensional subsimplex $E\subset\partial K$, there holds that
$$\int_E\phi_{K}ds=0.$$
\end{lemma}
\begin{proof}
Let $\theta_j=\theta_j(x),1\leq j\leq n+1$ denote the barycentric coordinates of $K$ associated to vertex $a_j$. For any integers $\alpha_j\geq 0,1\leq j\leq n+1$, one has
\begin{equation*}
  \int_K\theta^{\alpha_1}_1\theta^{\alpha_2}_2\cdots\theta^{\alpha_{n+1}}_{n+1}dx=\frac{\alpha_1!\alpha_2!\cdots\alpha_{n+1}!n!}{(\alpha_1+\alpha_2+\cdots+\alpha_{n+1}+n)!}|K|.
\end{equation*}
This leads to
\begin{equation*}
\begin{split}
 \int_K(x_i-M_i)(x_j-M_j)dx&=\int_K\sum^{n+1}_{p=1}(\theta_p-\frac{1}{n+1})x_{ip}\sum^{n+1}_{q=1}(\theta_q-\frac{1}{n+1})x_{jq}dx\\
 &=\frac{|K|}{(n+1)^2(n+2)}\left(\sum^{n+1}_{p=1}nx_{ip}x_{jp}-\sum_{p\neq q}x_{ip}x_{jq}\right)\\
 &=\frac{|K|}{(n+1)^2(n+2)}\sum_{p<q}(x_{ip}-x_{iq})(x_{jp}-x_{jq}).
 \end{split}
\end{equation*}
By the definition of $\phi_K$ in \eqref{bubblefunction}, this yields
\begin{equation*}
\begin{split}
  \frac{1}{|K|}\int_K\phi_Kdx&=\frac{n+2}{2}-\frac{1}{|K|}\frac{n(n+1)^2(n+2)}{2H}\frac{|K|}{(n+1)^2(n+2)}\\
  &\quad \times\sum_{i,j=1}^n\sum_{p<q}b_{ij}(x_{ip}
  -x_{iq})(x_{jp}-x_{jq})\\
  &=\frac{n+2}{2}-\frac{n}{2H}H\\
  &=1.
\end{split}
\end{equation*}
Given $n-1$ dimensional subsimplex $E\subset\partial K$, such that $\theta_1|_E\equiv0$. A similar equality holds
\begin{equation*}
  \int_E\theta^{\alpha_2}_2\cdots\theta^{\alpha_{n+1}}_{n+1}ds=\frac{\alpha_2!\cdots\alpha_{n+1}!(n-1)!}{(\alpha_2+\cdots+\alpha_{n+1}+n-1)!}|E|.
\end{equation*}A direct calculation yields
\begin{equation*}
\begin{split}
 \int_E(x_i-M_i)(x_j-M_j)ds=&\int_E\bigg(-\frac{x_{i1}}{n+1}+\sum^{n+1}_{p=2}(\theta_p-\frac{1}{n+1})x_{ip}\bigg)\\
 &\quad \times\bigg(-\frac{x_{j1}}{n+1}
 +\sum^{n+1}_{q=2}(\theta_q-\frac{1}{n+1})x_{jq}\bigg)ds\\
 =&\frac{|E|}{n(n+1)^2}\bigg(\sum^{n+1}_{p=1}nx_{ip}x_{jp}-\sum_{p\neq q}x_{ip}x_{jq}\bigg)\\
 =&\frac{|E|}{(n+1)^2(n+2)}\sum_{p<q}(x_{ip}-x_{iq})(x_{jp}-x_{jq}).
 \end{split}
\end{equation*}
This shows that
\begin{equation*}
 \int_E\phi_Kds=\frac{n+2}{2}|E|-\frac{n(n+1)^2(n+2)}{2H}\frac{|E|}{n(n+1)^2}H=0,
\end{equation*}
which completes the proof.
\end{proof}
Lemma \ref{Lemma:caculation} allows for the definition of the following bubble function space
$$
V_{\rm B}:=\{v\in L^2(\Omega):v|_K\in\sspan\{\phi_K\}\text{ for all }K\in\mathcal{T}\}.
$$
The GCR element space $V_{\rm GCR}$ is then defined by
\begin{equation}\label{GCRelement}
  V_{\rm GCR}:=V_{\rm CR}+V_{\rm B}.
\end{equation}
If $A(x)\equiv 1$, then $b_{ij}=\delta_{ij}$, $H=\sum_{p<q}|a_p-a_q|^2$ and
\begin{equation*}
\phi_K= \frac{n+2}{2}-\frac{n(n+1)^2(n+2)}{2H}\sum^{n}_{i=1}(x_i-M_i)^2\in{\rm ECR}(K).
\end{equation*}
Hence, in this case, $V_{\rm GCR}=V_{\rm ECR}$. The GCR element has the following important property.
\begin{lemma}
\label{edgeconstant}
Given $v\in V_{\rm GCR}$, $\bar{A}\nabla v\cdot\nu_E$ is a constant on $E$ for all $E\in\mathcal{E}$.
\end{lemma}
\begin{proof}
Given $E\in\mathcal{E}$, $x\cdot\nu_E$ is a constant on $E$. The fact that $\bar{B}$ is the inverse of $\bar{A}$, \eqref{bubblefunction} and \eqref{GCRelement} imply that $\bar{A}\nabla v\cdot\nu_E$ is a constant on $E$.
\end{proof}
\subsection{The GCR element for second order elliptic boundary value problems}
The generalized  Crouzeix-Raviart element method of Problem \eqref{ellipticproblem} finds  $u_{\rm GCR}\in V_{\rm GCR}$ such that
\begin{equation}\label{GCRelliptic}
(A\na_{\rm NC} u_{\rm GCR}, \na_{\rm NC} v)=(f,v)\quad\text{for any }v\in V_{\rm GCR}.
\end{equation}
Since  $\int_{E}[v]dE=0\text{ for all }E\in\mathcal{E}(\Omega)$ and $\int_EvdE=0\text{ for all }E\in\mathcal{E}(\partial\Omega)$. From the theory of \cite{HuMa2014}, there holds that
\begin{equation*}
  \|\nabla_{\rm NC}(u-u_{\rm GCR})\|\lesssim\|\nabla u-\Pi_0\nabla u\|+osc(f),
\end{equation*}
where $\Pi_0$ denotes the piecewise constant projection,  and the oscillation of data reads
$$osc(f)=\left(\sum_{K\in\mathcal{T}}h_K^2\big[\inf_{\bar{f}\in P_r(K)}\|f-\bar{f}\|^2_{L^2(K)}\big]\right)^{1/2}$$
with  arbitrary $r\geq 0$. The optimal convergence of the GCR element follows immediately.
\begin{remark}Thanks to the definition of \eqref{GCRelement}, $u_{\rm GCR}$ can be written as $u_{\rm GCR}=u_{\rm CR}+u_{\rm B}$, where $u_{\rm CR}\in V_{\rm CR}$ and $u_{\rm B}\in V_{\rm B}$.
When $A$ is a piecewise constant matrix-valued function, an integration by parts yields the following orthogonality:
\begin{equation}\label{Orthogonalitysecond}
  (A\nabla u_{\rm CR},\nabla\phi_K)_{L^2(K)}=(-\div(A\nabla u_{\rm CR}),\phi_K)_{L^2(K)}+\sum_{E\subset\partial K}\int_EA\nabla u_{\rm CR}\cdot\nu_E\phi_Kds=0.
\end{equation}
This leads to
\begin{equation}\label{bubblesoving}
(A\nabla u_{\rm B},\nabla\phi_K)_{L^2(K)}=(f,\phi_K)_{L^2(K)}\quad\text{for any }K\in\mathcal{T},
\end{equation}
and
\begin{equation}\label{CRproblem}
  (A\nabla_{\rm NC} u_{\rm CR},\nabla_{\rm NC} v)=(f,v)\quad\text{for any }v\in V_{\rm CR}.
\end{equation}
Consequently, $u_{\rm CR}$ is the discrete solution of Problem \eqref{ellipticproblem} by the CR element.
Hence we can solve the GCR element equation \eqref{GCRelliptic} by solving \eqref{bubblesoving} on each $K$ and \eqref{CRproblem} for the CR element, respectively.
For general cases, the orthogonality  \eqref{Orthogonalitysecond} does not hold. However, $u_{\rm B}$ can be eliminated a prior by a static condensation procedure.
\end{remark}
\subsection{The GCR element for second order elliptic eigenvalue problems}
\label{GCReigenvalue}
We consider the discrete
eigenvalue problem: Find $(\lam_{\rm GCR}, u_{\rm GCR})\in \R\x V_{\rm GCR}$ such that
\begin{equation}\label{discreteeigentotal}
\begin{split}
(A\nabla_{\rm NC}u_{\rm GCR}, \nabla_{\rm NC}v)&=\lam_{\rm GCR} (u_{\rm GCR},v) \text{ for any }
v\in V_{\rm GCR} \text{ and }  \|u_{\rm GCR}\|=1.
\end{split}
\end{equation}

Let $Z=\dim  V_{\rm GCR}$. The discrete problem
\eqref{discreteeigentotal} admits a sequence of discrete eigenvalues
$$
0<\lam_{1,\rm GCR}\leq\lam_{2,\rm GCR}\leq\cdots\leq\lam_{Z,\rm GCR},
$$
and corresponding eigenfunctions
$$
u_{1,\rm GCR}, u_{2,\rm GCR}, \cdots, u_{Z,\rm GCR}\,.
$$
Define the discrete counterpart of $E_{\ell}$ by
\begin{equation}\label{discretespan}
E_{\ell, \rm GCR}=\sspan\{u_{1,\rm GCR},u_{2,\rm GCR}, \cdots, u_{\ell,\rm GCR}\}.
\end{equation}
Then,  we have the following discrete Rayleigh-Ritz principle:
\begin{equation}\label{disrceteRayleigh}
\lam_{k,\rm GCR}=\min\limits_{\dim V_{k}=k, V_{k}\subset
V_{\rm GCR}}\max\limits_{v\in V_{k}}\frac{(A\nabla_{\rm NC}v,\nabla_{\rm NC}v)}{(
v,v)}=\max\limits_{u\in E_{k,\rm GCR}}\frac{(A\nabla_{\rm NC}u,\nabla_{\rm NC}u)}{(
u,u)}.
\end{equation}

According to the theory of nonconforming eigenvalue approximations \cite{BabuskaOsborn1991,HuHuangLin2010}, the following a priori estimate holds true.
\begin{lemma}Let $u$ be eigenfunctions of Problem \eqref{eigen}, and $u_{\rm GCR}$ be discrete eigenfunctions of Problem \eqref{GCReigenvalue}. Suppose $u\in H^1_0(\Omega)\cap H^{1+s}(\Omega)$ with $0<s\leq1$. Then,
\begin{equation}\label{eqorder}
\|u-u_{\rm GCR}\|+h^s\|\nabla_{\rm NC}(u-u_{\rm GCR})\|_A\lesssim h^{2s}|u|_{1+s}.
\end{equation}
\end{lemma}

 We introduce the interpolation operator $\Pi_{\rm GCR}:H^1_0(\Omega)\rightarrow V_{\rm GCR}$ by
\begin{equation}\label{inteGCR}
  \begin{split}
  \int_E\Pi_{\rm GCR}vds=\int_Evds\text{ for any }E\in\cE,\\
  \int_K\Pi_{\rm GCR}vdx=\int_Kvdx\text{ for any }K\in\mathcal{T}.
  \end{split}
\end{equation}
Given $w\in V_{\rm GCR}$, an integration by parts yields that
\begin{equation*}
\begin{split}
(\bar{A}\nabla_{\rm NC}(v-\Pi_{\rm GCR}v),\nabla_{\rm NC} w)&=-(v-\Pi_{\rm GCR}v,\div_{\rm NC}(\bar{A}\nabla_{\rm NC} w))\\
&+\sum_{K\in\mathcal{T}}\sum_{E\subset\partial K}\int_E(v-\Pi_{\rm GCR}v)\bar{A}\nabla w\cdot\nu_E ds.
\end{split}
\end{equation*}
Since $\div_{\rm NC}(\bar{A}\nabla_{\rm NC} w)$ is a piecewise constant on $\Omega$ and Lemma \ref{edgeconstant} proves that $\bar{A}\nabla w\cdot\nu_E$ is a constant on $n-1$ dimensional subsimplex $E$, for any $v\in H^1_0(\Omega)$, the following orthogonality holds true
\begin{equation}\label{projectionGradu}
(\bar{A}\nabla_{\rm NC}(v-\Pi_{\rm GCR}v),\nabla_{\rm NC} w)=0\quad\text{for any }w\in V_{\rm GCR}.
\end{equation}
This orthogonality is important in providing lower bounds for eigenvalues, see more details in the following two sections. Moreover, this yields
\begin{equation}\label{normortho}
 \|\nabla_{\rm NC}\Pi_{\rm GCR}v\|_{\bar{A}}^2+ \|\nabla_{\rm NC}(v-\Pi_{\rm GCR}v)\|_{\bar{A}}^2=\|\nabla v\|_{\bar{A}}^2.
\end{equation}
\section{Asymptotic lower bounds for eigenvalues}
We assume $A$ is a  piecewise constant matrix-valued function in this section. Following the theory of \cite{HuHuangLin2010}, we prove that the eigenvalues produced by the GCR element are lower bounds when the meshsize is small enough.

Let $(\lam,u)$ and $(\lam_{\rm GCR},u_{\rm GCR})$ be solutions of \eqref{eigen} and \eqref{discreteeigentotal}, respectively. First, note that  $u-\Pi_{\rm GCR}u$ has vanishing mean on each $K\in\mathcal{T}$. It follows from the Poincar$\acute{\rm e}$ inequality that
$$
\|u-\Pi_{\rm GCR}u\|\lesssim h\|\nabla_{\rm NC}(u-\Pi_{\rm GCR}u)\|.
$$
Suppose $u\in H^{1+s}(\Omega),0<s\leq1$. Following from the usual interpolation theory, there holds that
\begin{equation}\label{interpolationestimate}
\|u-\Pi_{\rm GCR}u\|\lesssim h^{1+s}|u|_{1+s}.
\end{equation}
\begin{theorem}
\label{asymptoticLowerbound}
Suppose that $A$ is a piecewise constant matrix-valued function. Assume that $u\in H^1_0(\Omega)\cap H^{1+s}(\Omega)$ with $0<s\leq1$ and that $h^{2s}\lesssim\|\nabla_{\rm NC}(u-u_{\rm GCR})\|_A^2$
. Then,
$$
\lam_{\rm GCR}\leq\lam,
$$
provided that $h$ is small enough.
\end{theorem}
\begin{proof}
Since $A$ is a piecewise constant matrix-valued function, $A=\bar{A}$, and $\bar{A}$ in \eqref{projectionGradu} can be replaced by $A$. A similar argument in \cite{Armentano,HuHuangLin2010,Zhang2007} proves
\begin{equation}\label{decomposition}
\begin{split}
  \lam-\lam_{\rm GCR}=&\|\nabla_{\rm NC}(u-u_{\rm GCR})\|_A^2-\lam_{\rm GCR}\|\Pi_{\rm GCR}u-u_{\rm GCR}\|^2\\
  &+\lam_{\rm GCR}(\|\Pi_{\rm GCR}u\|^2-\|u\|^2).
  \end{split}
\end{equation}
The triangle inequality, \eqref{eqorder} and \eqref{interpolationestimate} yield
$$
\lam_{\rm GCR}\|\Pi_{\rm GCR}u-u_{\rm GCR}\|^2\lesssim h^{4s}+h^{2+2s}\lesssim h^{4s}.
$$
It follows from the definition of the interpolation operator $\Pi_{\rm GCR}$, see \eqref{inteGCR}, that
\begin{equation*}
  \begin{split}
 \lam_{\rm GCR}\left( \|\Pi_{\rm GCR}u\|^2-\|u\|^2\right)&=\lam_{\rm GCR}(\Pi_{\rm GCR}u-u,\Pi_{\rm GCR}u+u)\\
  &=\lam_{\rm GCR}(\Pi_{\rm GCR}u-u,\Pi_{\rm GCR}u+u-\Pi_0(\Pi_{\rm GCR}u+u))\\
  &\lesssim h\|\Pi_{\rm GCR}u-u\|\|\nabla_{\rm NC}(\Pi_{\rm GCR}u+u)\|\\
  &\lesssim h^{2+s}.
  \end{split}
\end{equation*}
The above two estimates and the saturation condition  $h^{2s}\lesssim\|\nabla_{\rm NC}(u-u_{\rm GCR})\|_A^2$ imply that the second and third terms on the right-hand of \eqref{decomposition} are of higher order than the first term. This completes the proof.
\end{proof}
\begin{remark}
Hu et al. analyzed the saturation condition in \cite{HuHuangLin2010}. If the eigenfunctions $u\in H^{1+s}(\Omega)$ with $0<s<1$, it was proved that there exist meshes such that the saturation condition $h^{s}\lesssim\|\nabla_{\rm NC}(u-u_{\rm GCR})\|_A$ holds. In the following lemmas, we will prove the saturation condition $h\lesssim\|\nabla_{\rm NC}(u-u_{\rm GCR})\|_A$ provided that $u\in H^2(\Omega)$.

For simplicity, we prove it in two dimensions for the GCR element.
\end{remark}
\begin{lemma}
\label{sec3:lemma1}
Given $0\neq u\in H^1_0(\Omega)\cap H^2(\Omega)$, for any triangulation $\mathcal{T}$, there holds that
\begin{equation}\label{sec3:eq1}
\sum_{K\in\mathcal{T}}\left(\|\frac{\partial^2u}{\partial x_1^2}-\frac{b_{11}}{b_{22}}\frac{\partial^2u}{\partial x_2^2}\|^2_{L^2(K)}+\|\frac{\partial^2u}{\partial x_1\partial x_2}-\frac{b_{12}}{b_{11}}\frac{\partial^2u}{\partial x_1^2}\|^2_{L^2(K)}\right)>0.
\end{equation}
\end{lemma}
\begin{proof}
If \eqref{sec3:eq1} would not hold, then, for any $K\in\mathcal{T}$, $\|\frac{\partial^2u}{\partial x_1^2}-\frac{b_{11}}{b_{22}}\frac{\partial^2u}{\partial x_2^2}\|_{L^2(K)}=0$. Since $\bar{B}|_K$ is positive-definite, we have $b_{ii}>0,i=1,2$.
Hence $u$ should be of the form
$$u|_K(x_1,x_2)=\phi(x_1-\sqrt{\frac{b_{22}}{b_{11}}}x_2)+\psi(x_1+\sqrt{\frac{b_{22}}{b_{11}}}x_2),
$$
where $\phi (\cdot)$ and $\psi(\cdot)$ are two univariate functions.
Since $\|\frac{\partial^2u}{\partial x_1\partial x_2}-\frac{b_{12}}{b_{11}}\frac{\partial^2u}{\partial x_1^2}\|_{L^2(K)}=0$, we have
$$
(\sqrt{b_{11}b_{22}}+b_{12})\phi^{''}(x_1-\sqrt{\frac{b_{22}}{b_{11}}}x_2)=(\sqrt{b_{11}b_{22}}-b_{12})\psi^{''}(x_1+\sqrt{\frac{b_{22}}{b_{11}}}x_2).
$$
This yields that $\phi^{''}=\frac{\sqrt{b_{11}b_{22}}-b_{12}}{\sqrt{b_{11}b_{22}}+b_{12}}\psi^{''}\equiv C$ for some constant $C$. It's straightforward to derive that \begin{equation*}
\begin{split}
u|_K=&c_0+c_1(x_1-\sqrt{\frac{b_{22}}{b_{11}}}x_2)+c_2(x_1-\sqrt{\frac{b_{22}}{b_{11}}}x_2)^2+c_3(x_1+\sqrt{\frac{b_{22}}{b_{11}}}x_2)
\\
&+\frac{\sqrt{b_{11}b_{22}}+b_{12}}{\sqrt{b_{11}b_{22}}-b_{12}}c_2(x_1+\sqrt{\frac{b_{22}}{b_{11}}}x_2)^2\\
=&c_0+c_1(x_1-\sqrt{\frac{b_{22}}{b_{11}}}x_2)+c_3(x_1+\sqrt{\frac{b_{22}}{b_{11}}}x_2)\\
&+\frac{c_2\sqrt{b_{22}}}{\sqrt{b_{11}}(\sqrt{b_{11}b_{22}}-b_{12})}(b_{11}x_1^2+b_{22}x_2^2+2b_{12}x_1x_2),
\end{split}
\end{equation*}
for some interpolation parameters $c_0,c_1,c_2,c_3$. Furthermore, since $b_{11}b_{22}-b_{12}^2>0$, $b_{11}x_1^2+b_{22}x_2^2+2b_{12}x_1x_2$ can't be a linear function on any one dimensional subsimplex of $K$. The homogenous boundary condition and the continuity indicate that $u\in V_{\rm CR}\cap H^1_0(\Omega)\cap H^2(\Omega)$. This implies $u\equiv0$, which contradicts with $u\neq 0$.
\end{proof}
\begin{remark}
When the domain is a rectangle, the saturation condition was analyzed in \cite{HuHuangLin2010}. The theory of \cite{LinXie} does not cover both the ECR and GCR elements, see Corollary 3.3 therein.
\end{remark}
In order to achieve the desired result, we shall use the operator defined in \cite{HuHuangLin2010}. Given any $K\in\mathcal{T}$, define $J_{2,K}v\in P_2(K)$ by
\begin{equation*}
  \int_K\nabla^pJ_{2,K}vdx= \int_K\nabla^pvdx,\quad p=0,1,2
\end{equation*}
for any $v\in H^2(K)$. Note that the operator $J_{2,K}$ is well-defined. Since $\int_K\nabla^p(v-J_{2,K}v)dx=0$ with $p=0,1,2$, there holds that
\begin{equation}\label{operatorSaturationestimate}
  \|\nabla^{p_1}(v-J_{2,K})v\|_{L^2(K)}\lesssim h_K^{p_2-p_1} \|\nabla^{p_2}(v-J_{2,K})v\|_{L^2(K)}\text{ for any }0\leq p_1\leq p_2\leq 2.
\end{equation}
  Finally, define the global operator $J_2$ by
  \begin{equation}\label{operatorSaturation}
    J_2|_K=J_{2,K}\quad\text{for any } K\in\mathcal{T}.
  \end{equation}
It follows from the definition of $J_{2,K}$ in \eqref{operatorSaturation} that
\begin{equation*}
 \nabla^2J_{2,K}v=\Pi_0 \nabla^2v.
\end{equation*}
Since piecewise constant functions are dense in the space $L^2(\Omega)$,
\begin{equation}\label{approx}
  \|\nabla^2_{\rm NC}(v-J_2v)\|\rightarrow 0\text{ when }h\rightarrow 0.
\end{equation}
\begin{lemma}
Suppose that $A$ is a piecewise constant matrix-valued function. Suppose that $u\in H^1_0(\Omega)\cap H^2(\Omega)$, there holds the following saturation condition:
\begin{equation*}
  h\lesssim\|\nabla_{\rm NC}(u-u_{\rm GCR})\|_A
\end{equation*}
\end{lemma}
\begin{proof}
Since $A$ is piecewise constant, when $h$ is small enough, for any $K\in\mathcal{T}$, $A|_K$ is constant. According to Lemma \ref{sec3:lemma1},
there exists constant $\alpha>0$ such that
\begin{equation*}
\alpha<\sum_{K\in\mathcal{T}}\left(\|\frac{\partial^2u}{\partial x_1^2}-\frac{b_{11}}{b_{22}}\frac{\partial^2u}{\partial x_2^2}\|^2_{L^2(K)}+\|\frac{\partial^2u}{\partial x_1\partial x_2}-\frac{b_{12}}{b_{11}}\frac{\partial^2u}{\partial x_1^2}\|^2_{L^2(K)}\right).
\end{equation*}
The fact that $u\in V_{\rm GCR}$ plus \eqref{twoDex} and \eqref{GCRelement} yield that
\begin{equation*}
  \sum_{K\in\mathcal{T}}\left(\|\frac{\partial^2u_{\rm GCR}}{\partial x_1^2}-\frac{b_{11}}{b_{22}}\frac{\partial^2u_{\rm GCR}}{\partial x_2^2}\|^2_{L^2(K)}+\|\frac{\partial^2u_{\rm GCR}}{\partial x_1\partial x_2}-\frac{b_{12}}{b_{11}}\frac{\partial^2u_{\rm GCR}}{\partial x_1^2}\|^2_{L^2(K)}\right)=0.
\end{equation*}
Let $J_2$ be defined as in \eqref{operatorSaturation}. It follows from the triangle inequality and the piecewise inverse estimate that
\begin{equation*}
  \begin{split}
 \alpha<\sum_{K\in\mathcal{T}}&\left(\|\frac{\partial^2(u-u_{\rm GCR})}{\partial x_1^2}-\frac{b_{11}}{b_{22}}\frac{\partial^2(u-u_{\rm GCR})}{\partial x_2^2}\|^2_{L^2(K)}\right.\\
 &\left.+\|\frac{\partial^2(u-u_{\rm GCR})}{\partial x_1\partial x_2}-\frac{b_{12}}{b_{11}}\frac{\partial^2(u-u_{\rm GCR})}{\partial x_1^2}\|^2_{L^2(K)}\right)\\
  \leq2\sum_{K\in\mathcal{T}}&\left(\|\frac{\partial^2(u-J_2u)}{\partial x_1^2}-\frac{b_{11}}{b_{22}}\frac{\partial^2(u-J_2u)}{\partial x_2^2}\|^2_{L^2(K)}\right.\\
  &+\|\frac{\partial^2(u-J_2u)}{\partial x_1\partial x_2}-\frac{b_{12}}{b_{11}}\frac{\partial^2(u-J_2u)}{\partial x_1^2}\|^2_{L^2(K)}\\
  &+\|\frac{\partial^2(J_2u-u_{\rm GCR})}{\partial x_1^2}-\frac{b_{11}}{b_{22}}\frac{\partial^2(J_2u-u_{\rm GCR})}{\partial x_2^2}\|^2_{L^2(K)}\\
 &\left.+\|\frac{\partial^2(J_2u-u_{\rm GCR})}{\partial x_1\partial x_2}-\frac{b_{12}}{b_{11}}\frac{\partial^2(J_2u-u_{\rm GCR})}{\partial x_1^2}\|_{L^2(K)}^2\right)\\
  \lesssim\|\nabla^2_{\rm NC}&(u-J_2u)\|^2+h^{-2}\|\nabla_{\rm NC}(J_2u-u_{\rm GCR})\|^2.
  \end{split}
\end{equation*}
The estimate of \eqref{operatorSaturationestimate} and the triangle inequality lead to
\begin{equation*}
   1\lesssim\|\nabla^2_{\rm NC}(u-J_2u)\|^2+h^{-2}\|\nabla_{\rm NC}(u-u_{\rm GCR})\|^2.
\end{equation*}
Finally it follows from \eqref{approx} that
\begin{equation*}
  h^2\lesssim\|\nabla_{\rm NC}(u-u_{\rm GCR})\|^2
\end{equation*}
when the meshsize is small enough, which completes the proof.
\end{proof}
\section{Guaranteed lower bounds for eigenvalues}
In practice, it is not easy to check whether the meshsize $h$ is small enough in Theorem \ref{asymptoticLowerbound}. In this section, we propose a new method to provide  guaranteed lower bounds for eigenvalues. We follow the idea of \cite{LiuOishi2013} and \cite{CarstensenGallistl2013,CarstensenGedickel2013} and generalize it to general second order elliptic operators. We first present some constants about the matrix-valued function $A$, which might be depend on $h$. For any $v\in H^1_0(\Omega)+V_{\rm GCR}$, there exist $C_A$, $C_{\bar{A}}$, $C_{\bar{A},A}$ and $C_\infty$ such that
\begin{equation}\label{constant1}
\|\nabla_{\rm NC}v\|\leq C_A\|\nabla_{\rm NC}v\|_A,
\end{equation}
\begin{equation}\label{constant2}
\|\nabla_{\rm NC}v\|\leq C_{\bar{A}}\|\nabla_{\rm NC}v\|_{\bar{A}},
\end{equation}
\begin{equation}\label{constant3}
\|\nabla_{\rm NC}v\|_{\bar{A}}\leq C_{\bar{A},A}\|\nabla_{\rm NC}v\|_A,
\end{equation}
\begin{equation}\label{constant4}
  \|(A-\bar{A})\nabla_{\rm NC}v\|\leq C_\infty h\|\nabla_{\rm NC}v\|.
\end{equation}
Define $\eta_1:=C_{\bar{A}}C_{\bar{A},A}$ and $\eta_2:=C_\infty C_{\bar{A}}C_AC_{\bar{A},A}$.

The following Poincar$\acute{\rm e}$ inequality can be found in \cite{ChavelFeldman1977} .
\begin{lemma}
\label{Chavelpoincare}
Given $K\in\mathcal{T}$, let $w\in H^1(K)$ be a function with vanishing mean. Then
\begin{equation*}
  \|w\|_{L^2(K)}\leq\frac{h_K}{\pi}\|\nabla w\|_{L^2(K)}.
\end{equation*}
\end{lemma}
\begin{remark}
\label{poincarebetter}
Let $j_{1,1}= 3.8317059702$ be the first positive root of the Bessel function of the first kind. In two dimensions, the following improved Poincar$\acute{\rm e}$ inequality holds from \cite{LaugesenSiudeja},
\begin{equation*}
  \|w\|_{L^2(K)}\leq\frac{h_K}{j_{1,1}}\|\nabla w\|_{L^2(K)}.
\end{equation*}
\end{remark}
Thanks to the second equation of \eqref{inteGCR}, for any $v\in H^1(K)$, there holds that
\begin{equation}\label{Poincare}
  \|v-\Pi_{\rm GCR}v\|_{L^2(K)}\leq \frac{h_K}{\pi}\|\nabla(v-\Pi_{\rm GCR}v)\|_{L^2(K)}.
\end{equation}
\begin{theorem}
\label{guranteelowerbound} Let $\lambda_\ell$ and $\lambda_{\ell,\rm GCR}$ be the $\ell-$th eigenvalues of \eqref{eigen} and \eqref{discreteeigentotal}, respectively. The meshsize of the triangulation is chosen to
be sufficiently small such that
\begin{equation}
\label{hconstrain}
h<\frac{\pi}{\eta_1\sqrt{\lam_\ell}}.
\end{equation}Then, there holds that, for any $0<\beta<1$
\begin{equation}
\label{lowerbound}
 \frac{\lambda_{\ell,\rm GCR}}{1+\frac{\lambda_{\ell,\rm GCR}^2C_A^4h^4}{4\pi^2(\beta\pi^2+\lambda_{\ell,\rm GCR}C_A^2h^2)}+\frac{\eta_2^2h^2}{1-\beta}}\leq \lambda_\ell.
\end{equation}
\end{theorem}
\begin{proof}$E_\ell$ is defined in \eqref{ell}. For any $v=\sum^\ell_{k=1}c_iu_i\in E_\ell$, $\|v\|=1$. It's immediate to see that $\|\nabla v\|_A\leq\sqrt{\lam_\ell}$. Therefore, \eqref{Poincare}, the constant in \eqref{constant2}, the  property \eqref{normortho} for the interpolation operator,  and \eqref{constant3} imply that
\begin{equation}
\begin{split}
\|v-\Pi_{\rm GCR}v\|&\leq\frac{h}{\pi}\|\nabla_{\rm NC}(v-\Pi_{\rm GCR}v)\|\leq\frac{C_{\bar{A}}h}{\pi}\|\nabla_{\rm NC}(v-\Pi_{\rm GCR}v)\|_{\bar{A}}\\
&\leq\frac{C_{\bar{A}}h}{\pi}\|\nabla v\|_{\bar{A}}\leq\frac{C_{\bar{A}}C_{\bar{A},A}h}{\pi}\|\nabla v\|_A\leq\frac{\eta_1h}{\pi}\sqrt{\lam_\ell},
\end{split}
\end{equation}
where $\eta_1=C_{\bar{A}}C_{\bar{A},A}$. Due to the assumption $h<\frac{\pi}{\eta_1\sqrt{\lam_\ell}}$, there holds that
\begin{equation*}
  \|\Pi_{\rm GCR}v\|\geq1-\|v-\Pi_{\rm GCR}v\|\geq1-\frac{\eta_1h}{\pi}\sqrt{\lam_\ell}>0.
\end{equation*}
As $E_\ell$ is a $\ell-$dimensional space, $\Pi_{\rm GCR}E_\ell$ is also a $\ell-$dimensional space.

There exist real coefficients $\xi_1,\cdots,\xi_\ell$ with $\sum^\ell_{k=1}\xi^2_k=1$ such
that the maximiser of the Rayleigh quotient \eqref{disrceteRayleigh} in $\sspan\{\Pi_{\rm GCR}u_1,\cdots,\Pi_{\rm GCR}u_\ell\}$ is equal to $\sum^\ell_{k=1}\xi_k\Pi_{\rm GCR}u_k$. Therefore $v:=\sum^\ell_{k=1}\xi_ku_k$ satisfies
\begin{equation}\label{eqRQ}
  \lam_{\ell,\rm GCR}\leq\frac{\|\nabla_{\rm NC}\Pi_{\rm GCR}v\|^2_A}{\|\Pi_{\rm GCR}v\|^2}.
\end{equation}
An elementary manipulation yields the following decomposition
\begin{equation}\label{equ1}
\begin{split}
\|\nabla v\|_A^2&=\|\nabla_{\rm NC}(v-\Pi_{\rm GCR}v)\|_A^2+\|\nabla_{\rm NC}\Pi_{\rm GCR}v\|_A^2\\
&+2(A(\nabla_{\rm NC}(v-\Pi_{\rm GCR}v),\nabla_{\rm NC}\Pi_{\rm GCR}v).
\end{split}
\end{equation}
For the first term of \eqref{equ1}, it follows from \eqref{constant1} and \eqref{Poincare} that
\begin{equation}\label{term1}
 \|\nabla_{\rm NC}(v-\Pi_{\rm GCR}v)\|_A^2\geq\frac{\pi^2}{C_A^2h^2}\|v-\Pi_{\rm GCR}v\|^2.
\end{equation}
The second term of \eqref{equ1} can be analyzed by \eqref{eqRQ} as
\begin{equation}\label{term2}
\begin{split}
\|\nabla_{\rm NC}\Pi_{\rm GCR}v\|_A^2&\geq\lam_{\ell,\rm GCR}\|\Pi_{\rm GCR}v\|^2\\
&=\lambda_{\ell,\rm GCR}(\|v-\Pi_{\rm GCR}v\|^2+\|v\|^2-2(v-\Pi_{\rm GCR}v,v)).
\end{split}
\end{equation}
By the second equation of \eqref{inteGCR}, we have
 $$(v-\Pi_{\rm GCR}v,v)=(v-\Pi_{\rm GCR}v,v-\Pi_0v).$$
Since $\int_K\Pi_0vdx=\int_Kvdx$, the same estimate of \eqref{Poincare} holds true for $\Pi_0$. \eqref{constant1} and the Young inequality reveal for any $\delta_1>0$ that
\begin{equation*}
\begin{split}
  (v-\Pi_{\rm GCR}v,v-\Pi_0v)&\leq\|v-\Pi_{\rm GCR}v\|\|v-\Pi_0v\|\leq\frac{h}{\pi}\|v-\Pi_{\rm GCR}v\|\|\nabla v\|\\
  &\leq\frac{C_Ah}{\pi}\|v-\Pi_{\rm GCR}v\|\|\nabla v\|_A\\
  &\leq\frac{C_A^2h^2}{2\pi^2}\delta_1\|v-\Pi_{\rm GCR}v\|^2+\frac{1}{2\delta_1}\|\nabla v\|_A^2.
  \end{split}
\end{equation*}
The third term of \eqref{equ1} has the following decomposition:
\begin{equation}\label{term3}
  \begin{split}
  2(A(\nabla_{\rm NC}(v-\Pi_{\rm GCR}v),\nabla_{\rm NC}&\Pi_{\rm GCR}v)=2(\bar{A}(\nabla_{\rm NC}(v-\Pi_{\rm GCR}v),\nabla_{\rm NC}\Pi_{\rm GCR}v)\\
  &+2((A-\bar{A})\nabla_{\rm NC}(v-\Pi_{\rm GCR}v),\nabla_{\rm NC}\Pi_{\rm GCR}v).
  \end{split}
\end{equation}
Thanks to \eqref{projectionGradu}, the first term in the above equation equals to zero. It remains to estimate the second term, which can be estimated by \eqref{constant1}--\eqref{constant4}, \eqref{normortho} and the Young inequality that
\begin{equation*}
\begin{split}
&2((A-\bar{A})\nabla_{\rm NC}(v-\Pi_{\rm GCR}v),\nabla_{\rm NC}\Pi_{\rm GCR}v)\\
&\leq2C_{\bar{A}}\|(A-\bar{A})\nabla_{\rm NC}(v-\Pi_{\rm GCR}v)\|\|\nabla_{\rm NC}\Pi_{\rm GCR}v\|_{\bar{A}}\\
&\leq2C_{\bar{A}}C_\infty h\|\nabla_{\rm NC}(v-\Pi_{\rm GCR}v)\|\|\nabla v\|_{\bar{A}}\\
&\leq2\eta_2h\|\nabla_{\rm NC}(v-\Pi_{\rm GCR}v)\|_A\|\nabla v\|_A\\
&\leq\delta_2\|\nabla_{\rm NC}(v-\Pi_{\rm GCR}v)\|_A^2+\frac{\eta_2^2h^2}{\delta_2}\|\nabla v\|_A^2,
 \end{split}
\end{equation*}
where $\eta_2=C_\infty C_{\bar{A}}C_AC_{\bar{A},A}$ and $\delta_2>0$ is arbitrary. By substituting \eqref{term1}--\eqref{term3} into \eqref{equ1}, we obtain, for any $0<\beta<1$, that
\begin{equation*}
\begin{split}
\lam_\ell&\geq\|\nabla v\|^2_A\geq\left(\beta\frac{\pi^2}{C_A^2h^2}+\lambda_{\ell,\rm GCR}-\lambda_{\ell,\rm GCR}\frac{C_A^2h^2\delta_1}{2\pi^2}\right)\|v-\Pi_{\rm GCR}v\|^2\\
&+(1-\beta-\delta_2)\|\nabla_{\rm NC}(v-\Pi_{\rm GCR}v)\|_A^2-\left(\frac{\lambda_{\ell,\rm GCR}}{2\delta_1}+\frac{\eta_2^2h^2}{\delta_2}\right)\|\nabla v\|_A^2+\lambda_{\ell,\rm GCR}\|v\|^2.
\end{split}
\end{equation*}
Let $\delta_1=\frac{2\pi^2(\beta\pi^2+\lambda_{\ell,\rm GCR}C^2_Ah^2)}{\lambda_{\ell,\rm GCR}C^4_Ah^4}$, $\delta_2=1-\beta$. This yields that
\begin{equation*}
 \begin{split}
 0&\leq\|\triangledown v\|^2_A(1+\frac{\lambda_{\ell,\rm GCR}}{2\delta_1}+\frac{\eta_2^2h^2}{\delta_2})-\lambda_{\ell,\rm GCR}\|v\|^2\\
 &\leq\lam_\ell(1+\frac{\lambda_{\ell,\rm GCR}}{2\delta_1}+\frac{\eta_2^2h^2}{\delta_2})-\lambda_{\ell,\rm GCR}.
\end{split}
\end{equation*}
This completes the proof.
\end{proof}
\begin{remark}
When $A$ is a piecewise constant matrix-valued function, \eqref{lowerbound} yields that
\begin{equation}
 \frac{\lambda_{\ell,\rm GCR}}{1+\frac{\lambda_{\ell,\rm GCR}^2C_A^4h^4}{4\pi^2(\pi^2+\lambda_{\ell,\rm GCR}C_A^2h^2)}}\leq \lambda_\ell.
\end{equation}
Due to Remark \ref{poincarebetter}, $\pi$ can be replaced by $j_{1,1}$ in two dimensions. For the Laplace operator in two dimensions considered  in \cite{CarstensenGedickel2013}, as we shall find in Section 7, the guaranteed lower bounds of this paper are more accurate than those \cite{CarstensenGedickel2013}, see \eqref{CRlowerbound}
 below.  In addition, the meshsize condition
 \eqref{hconstrain} for this case becomes
 $$
 h<\frac{j_{1,1}}{\sqrt{\lam_\ell}}
 $$
 which improves largely that used in \cite{CarstensenGedickel2013} which reads
 $$
 h<\frac{\sqrt{1+1/\ell}-1}{\kappa\sqrt{\lam_\ell}}\text{ with } \kappa=\sqrt{1/48+1/j_{1,1}^2}.
 $$
 \end{remark}
\begin{remark}
Note that $\lam_\ell$ is unknown. In Section \ref{guaranteedupper}, we will propose a method to produce a guaranteed upper bound of $\lam_\ell$.
\end{remark}
\section{Asymptotic upper bounds for eigenvalues}
It is well-known that conforming finite element methods provide upper bounds for eigenvalues, but it needs to compute an extra eigenvalue problem. Here we present a simple postprocessing method to provide uppers bound for eigenvalues by the GCR element, see more details in \cite{HuHuangShen2011,ShenD}.

For any $v\in V_{\rm GCR}$, define the interpolation $\Pi_{\rm CR}:V_{\rm GCR}\rightarrow V_{\rm CR}$ by
\begin{equation*}
  \int_E\Pi_{\rm CR}vds=\int_Evds\text{ for any }E\in\cE.
\end{equation*}
It's straightforward to see that $v-\Pi_{\rm CR}v\in V_{\rm B}$. Furthermore, the standard interpolation theory of \cite{Ciarlet2002} gives
\begin{equation}\label{CRinterpolationproperty}
 \|v-\Pi_{\rm CR}v\|\lesssim h\|\nabla_{\rm NC}(v-\Pi_{\rm CR}v)\|\lesssim h^2\|\nabla^2_{\rm NC}v\|,
\end{equation}
An integration by parts leads to the following orthogonality:
\begin{equation}\label{functionOrthogonality}
(\nabla_{\rm NC}(v-\Pi_{\rm CR}v),\nabla_{\rm NC}\Pi_{\rm CR}v)=0.
\end{equation}

For any $v\in V_{\rm CR}$, define the interpolation $\Pi_{\rm c}:V_{\rm CR}\rightarrow V_{\rm c}:=V_{\rm CR}\cap H^1_0(\Omega)$ by
\begin{equation}\label{interpolationC}
(\Pi_{\rm c}v)(z)=\begin{cases}
0&z\in\partial\Omega,\\
\frac{1}{|\omega_z|}\sum_{K\in\omega_z}v|_K(z)&z\not\in\partial\Omega,
\end{cases}
\end{equation}
where $\omega_z$ is the union of elements containing vertex $z$, $|\omega_z|$ is the number of elements containing vertex $z$. The following lemma was proved in \cite{HuHuangShen2011,ShenD,YangHan2014}.
\begin{lemma}
\label{ConforminginterpolationProperty}
Let $v\in V_{\rm CR}$. For any $w\in H^1_0(\Omega)$, there holds that
\begin{equation*}
  \|v-\Pi_{\rm c}v\|\lesssim h\|\nabla_{\rm NC}(v-w)\|,
\end{equation*}
\begin{equation*}
  \|\nabla_{\rm NC}(v-\Pi_{\rm c}v)\|\lesssim \|\nabla_{\rm NC}(v-w)\|.
\end{equation*}
\end{lemma}
\eqref{CRinterpolationproperty} and Lemma \ref{ConforminginterpolationProperty} yield the following result.
\begin{corollary}
\label{corollary}
Let $u$ and $u_{\rm GCR}$ be eigenfunctions of \eqref{eigen} and \eqref{discreteeigentotal}, respectively. Suppose that $u\in H^{1+s}(\Omega),0<s\leq 1$. There holds that
\begin{equation*}
\|u_{\rm GCR}-\Pi_{\rm c}(\Pi_{\rm CR}u_{\rm GCR})\|\lesssim h^{1+s}|u|_{1+s},
\end{equation*}
\begin{equation*}
\|\nabla_{\rm NC}(u_{\rm GCR}-\Pi_{\rm c}(\Pi_{\rm CR}u_{\rm GCR}))\|_A\lesssim h^{s}|u|_{1+s}.
\end{equation*}
\end{corollary}
Define the Rayleigh quotient
$$
\lam_{c}=\frac{(A\nabla\Pi_{\rm c}(\Pi_{\rm CR}u_{\rm GCR}),\Pi_{\rm c}(\Pi_{\rm CR}u_{\rm GCR}))}{(\Pi_{\rm c}(\Pi_{\rm CR}u_{\rm GCR}),\Pi_{\rm c}(\Pi_{\rm CR}u_{\rm GCR}))}.
$$
\begin{theorem}
Suppose $(\lam,u)$ be eigenpairs of \eqref{eigen} and $u\in H^{1+s}(\Omega),0<s\leq1$, then
$$
|\lam-\lam_{\rm c}|\lesssim h^{2s}|u|_{1+s}.
$$
Moreover, $\lam_{\rm c}\geq\lam$ provided that $h$ is small enough.
\end{theorem}
\begin{proof}
The proof is similar to that of Theorem 3.4 in \cite{ShenD} and Theorem 4.1 in \cite{YangHan2014}. Let $w=\Pi_c(\Pi_{\rm CR}u_{\rm GCR})$. An elementary manipulation leads
\begin{equation}\label{eqthem}
  \begin{split}
  \|\nabla(u-w)\|^2_A&=(A\nabla(u-w),\nabla(u-w))=\lam+\|w\|^2\lam_c-2(A\nabla u,\nabla w)\\
  &=\lam+\|w\|^2\lam_c-2\lam(u,w)\\
  &=\|w\|^2(\lam_c-\lam)+\lam\|u-w\|^2.
  \end{split}
\end{equation}
Thanks to \eqref{eqorder} and Corollary \ref{corollary}, it holds that
\begin{equation}\label{est1}
 \|\nabla(u-w)\|_A\leq\|\nabla_{\rm NC}(u-u_{\rm GCR})\|_A+\|\nabla_{\rm NC}(u_{\rm GCR}-w)\|_A\lesssim h^s|u|_{1+s}
\end{equation}
and
\begin{equation}\label{est2}
 \|u-w\|\leq \|u-u_{\rm GCR}\|+\|u_{\rm GCR}-w\|\lesssim (h^{2s}+h^{1+s})|u|_{1+s}\lesssim h^{2s}|u|_{1+s}.
\end{equation}
On the other hand $\left|\|w\|-\|u\|\right|\leq\|u-w\|\lesssim h^{2s}|u|_{1+s}$. Hence $\|w\|$ is bounded. Substituting \eqref{est1} and \eqref{est2} into \eqref{eqthem} yields that
$$
|\lam-\lam_{\rm c}|\lesssim h^{2s}|u|_{1+s}.
$$
The following saturation condition holds, see \cite{HuHuangLin2010},
$$
h^s\lesssim \|\nabla(u-w)\|_A.
$$
Hence, when $h$ is small enough, $\|u-w\|$ is of higher order than $\|\nabla(u-w)\|_A$. This and \eqref{eqthem} yield that
$$
0\leq\|w\|^2(\lam_c-\lam),
$$
which completes the proof.
\end{proof}
\section{Guaranteed upper bounds for eigenvalues}
\label{guaranteedupper}
Because of the unknown of the exact eigenvalues, we need an upper bound of $\lam_\ell$ to guarantee \eqref{hconstrain}. Since $\lam_{c}$ is the upper bound of $\lam$ in the asymptotic sense. We propose a method to guarantee upper bounds for eigenvalues. Suppose $(\lam_\ell,u_{\ell})$ be the $\ell$-th eigenpair of \eqref{eigen} and $E_{\ell,\rm GCR}$ be defined in \eqref{discretespan}.
Define
\begin{equation}\label{eigenvalueGupper}
  \lam_{\ell,c}^m:=\sup_{v\in\Pi_{\rm c}(\Pi_{\rm CR}E_{\ell,\rm GCR})}\frac{(A\nabla v,\nabla v)}{(v,v)}.
\end{equation}

 \begin{lemma}
 Suppose that $u_\ell\in H^{1+s}(\Omega)$ with $0<s\leq 1$, then
 $$|\lam^m_{\ell,\rm c}-\lam_\ell|\lesssim h^{2s}|u|_{1+s}.
 $$
 \end{lemma}
 \begin{proof}
 Following the theory of \cite{BabuskaOsborn1991}, there holds that
\begin{equation*}
|\lam^m_{\ell,\rm c}-\lam_\ell|\lesssim\left(\inf_{v\in\Pi_{\rm c}(\Pi_{\rm CR}E_{\ell,\rm GCR})}\|\nabla(v-u_\ell)\|_A\right)^2\lesssim\|\nabla(\Pi_{\rm c}(\Pi_{\rm CR}u_{\ell,\rm GCR})-u_\ell)\|^2_A.
\end{equation*}
Hence, the above result and \eqref{est1} yield that
\begin{equation*}
|\lam^m_{\ell,\rm c}-\lam_\ell|\lesssim h^{2s}|u|_{1+s}.
\end{equation*}
This completes the proof.
 \end{proof}
 Assume that $\Pi_{\rm c}(\Pi_{\rm CR}E_{\ell,\rm GCR})$ is $\ell$-dimensional. The Rayleigh-Ritz principle \eqref{minmax} implies that $\lam_{\ell,c}^m$ is the upper bound of $\lam_\ell$. We propose some conditions in the following lemma to guarantee that $\Pi_{\rm c}(\Pi_{\rm CR}E_{\ell,\rm GCR})$ is $\ell$-dimensional.
\begin{lemma}
Suppose there exist computable constants $\beta_1$ and $\beta_2$  such that
$$
 \|v-\Pi_{\rm CR}v\|\leq\beta_1h\|\nabla_{\rm NC}(v-\Pi_{\rm CR}v)\|\text{ for any }v\in V_{\rm GCR},
 $$
 $$
 \|w-\Pi_{\rm c}w\|\leq\beta_2h\|\nabla_{\rm NC}w\|\text{ for any }w\in V_{\rm CR}.
 $$
Then, $\Pi_{\rm c}(\Pi_{\rm CR}E_{\ell,\rm GCR})$ is $\ell$-dimensional provided that
\begin{equation}\label{hconstrain2}
h<\frac{1}{(\beta_1+\beta_2)C_A\sqrt{\lam_{\ell,\rm GCR}}}.
\end{equation}
 \end{lemma}
 \begin{proof}
We adopt a similar argument in Theorem \ref{guranteelowerbound}. For any $v=\sum^\ell_{k=1}\xi_iu_{i,\rm GCR}$ and $\|v\|=1$, the triangle inequality yields
\begin{equation*}
\begin{split}
\|v-\Pi_{\rm c}(\Pi_{\rm CR}v)\|&\leq\|v-\Pi_{\rm CR}v\|+\|\Pi_{\rm CR}v-\Pi_{\rm c}(\Pi_{\rm CR}v)\|\\
&\leq\beta_1h\|\nabla_{\rm NC}(v-\Pi_{\rm CR}v)\|+\beta_2h\|\nabla_{\rm NC}\Pi_{\rm CR}v\|.
\end{split}
\end{equation*}
Due to \eqref{functionOrthogonality} and the constant in \eqref{constant1}, there holds the following estimate
\begin{equation*}
\begin{split}
  \|v-\Pi_{\rm c}(\Pi_{\rm CR}v)\|&\leq(\beta_1+\beta_2)h\|\nabla_{\rm NC}v\|\leq(\beta_1+\beta_2)C_Ah\|\nabla_{\rm NC}v\|_A\\
&\leq(\beta_1+\beta_2)C_Ah\sqrt{\lam_{\ell,\rm GCR}}.
\end{split}
\end{equation*}
Then, the condition for $h$ in \eqref{hconstrain2} yields
\begin{equation*}
  \|\Pi_{\rm c}(\Pi_{\rm CR}v)\|\geq 1- \|v-\Pi_{\rm c}(\Pi_{\rm CR}v)\|\geq1-(\beta_1+\beta_2)C_Ah\sqrt{\lam_{\ell,\rm GCR}}>0.
\end{equation*}
Hence, $\Pi_{\rm c}(\Pi_{\rm CR}E_{\ell,\rm GCR})$ is $\ell$-dimensional.
\end{proof}
\begin{remark}
\eqref{hconstrain2} is not a strict condition. Indeed, to obtain good approximation of the $\ell$-the eigenvalue $\lam_\ell$ by finite element methods, $\lam_\ell h^2\lesssim1$ is always required.
\end{remark}
We show that $\beta_1$ is computable. Note that $(v-\Pi_{CR}v)|_K\in\sspan\{\phi_K\}$, where $\phi_K$ is defined as in \eqref{bubblefunction}. For each $K\in\mathcal{T}$, we can find a positive constant $\beta_K$ such that
$$
\|\phi_K\|_{L^2(K)}\leq\beta_K\|\nabla\phi_K\|_{L^2(K)}.
$$
Then, we take
\begin{equation*}
  \beta_1=\frac{\max_{K\in\mathcal{T}}\{\beta_K\}}{h}.
\end{equation*}
There are several results concerning the constant for the interpolation operator $\Pi_{\rm CR}$ in two dimensions, see for instance \cite{CarstensenGallistl2013,MaoShi}. We present the result in \cite{CarstensenGallistl2013} as follows
\begin{equation*}
\|v-\Pi_{\rm CR}v\|_{L^2(K)}\leq\sqrt{j_{1,1}^{-2}+1/48}h_K\|\nabla(v-\Pi_{\rm CR}v)\|_{L^2(K)} \text{ for any }v\in H^1(K).
\end{equation*}
Hence we can choose $\beta_1=\sqrt{j_{1,1}^{-2}+1/48}\approx0.2984$ in two dimensions. As for any dimension, we give the constant for the interpolation operator by following the arguments in \cite{CarstensenGallistl2013}.
\begin{lemma}
Given $K\in\mathcal{T}$, let $f\in H^1(K)$ be a function with vanishing mean on any $n-1$ dimensional subsimplex $E\subset\partial{K}$. Then, there holds that
\begin{equation}\label{integralaverage}
\begin{split}
\left|\frac{1}{|K|}\int_Kf\diff x\right|\leq\frac{1}{\sqrt{2n(n+1)(n+2)}|K|^{1/2}}h_K\|\nabla f\|_{L^2(K)}.
\end{split}
\end{equation}
\end{lemma}
\begin{proof}
Let the centroid of $K$ be $M:={\rm mid}(K)$ and vertices $a_p,1\leq p\leq n+1$. The proof follows the trace identity,
\begin{equation}\label{eqtrace}
  \int_K\nabla f\cdot(x-M)dx=\int_{\partial K}f(x-M)\cdot\nu d s-\int_Kf\div(x-M) d x.
\end{equation}
Herein we use the fact that $(x-M)\cdot\nu$ is constant on any $n-1$ dimensional subsimplex $E\subset\partial{K}$ and $
\int_Efd s=0
$.
This yields that
\begin{equation}\label{eq6:1}
\begin{split}
\left|\int_Kfdx\right|&=\left|\frac{1}{n}\int_{K}(x-M)\cdot\nabla fd x\right|\\
&\leq\frac{1}{n}\|x-M\|_{L^2(K)}\|\nabla f\|_{L^2(K)}.
\end{split}
\end{equation}
A similar calculation as in Lemma \ref{Lemma:caculation} shows that
\begin{equation*}
 \|x-M\|^2_{L^2(K)}=\frac{|K|}{(n+1)^2(n+2)}\sum_{p<q}|a_p-a_q|^2\leq\frac{n|K|}{2(n+1)(n+2)}h^2_K.
\end{equation*}
Substituting the above result into \eqref{eq6:1} completes the proof.
\end{proof}
\begin{lemma}
For any $v\in H^1(K)$, it holds that
\begin{equation}\label{Poincare-Friedrichs Inequality}
\|v-\Pi_{\rm CR}v\|_{L^2(K)}\leq\kappa h_K\|\nabla(v-\Pi_{\rm CR}v)\|_{L^2(K)},
\end{equation}
where
\begin{equation}
\kappa=\sqrt{\pi^{-2}+\frac{1}{2n(n+1)(n+2)}}.
\end{equation}
\end{lemma}
\begin{proof}

Let $f=v-\Pi_{\rm CR}v$. The function $f$ satisfies, for any $n-1$ dimensional subsimplex $E\subset\partial K$,
\begin{equation*}
\int_Efds=0.
\end{equation*}
Let $f_K=\frac{1}{|K|}\int_Kfdx$ denote the integral mean on $K$, which leads to
\begin{equation*}
\|f\|^2_{L^2(K)}=\|f-f_K\|^2_{L^2(K)}+|K|f^2_K.
\end{equation*}
Lemma \ref{Chavelpoincare} plus \eqref{integralaverage} reveal
\begin{equation*}
\|f\|^2_{L^2(K)}\leq(\pi^{-2}+\frac{1}{2n(n+1)(n+2)})h^2_K\|\nabla f\|_{L^2(K)}^2,
\end{equation*}
which completes the proof.
\end{proof}
Hence we can choose
\begin{equation}\label{estimateofbeta1}
 \beta_1=\begin{cases}
 \sqrt{j_{1,1}^{-2}+\frac{1}{48}}\approx0.2984&n=2,\\
 \sqrt{\pi^{-2}+\frac{1}{2n(n+1)(n+2)}}&n\geq 3.
 \end{cases}
\end{equation}
Next, we analyze the computable constant $\beta_2$. To this end, we define
\begin{equation}\label{rate}
\xi=\max_{K,K'\in\mathcal{T}}\frac{|K'|}{|K|},
\end{equation}
and
\begin{equation}\label{elementnumber}
N=\max_{z\in\mathcal{V}}|\omega_z|,
\end{equation}
where $\mathcal{V}$ denotes the set of all the vertices of $\mathcal{T}$ and $|\omega_z|$ denotes the number of elements containing vertex $z$.
\begin{lemma}
For any $w\in V_{\rm CR}$, it holds that
\begin{equation*}
\|w-\Pi_{\rm c}w\|\leq\frac{(n-1)N\sqrt{\xi}}{n}h\|\nabla_{\rm NC}w\|.
\end{equation*}
\end{lemma}
\begin{proof}
Given element $K\in\mathcal{T}$, let $a_p,1\leq p\leq n+1$ be its vertices and $\theta_p$ be the corresponding barycentric coordinates.
 Then,
 $$w|_K=\sum^{n+1}_{p=1} w|_K(a_p)\theta_p\text{ and }
(\Pi_{\rm c}w)|_K=\sum^{n+1}_{p=1}\bar{w}_p\theta_p,$$ where
 \begin{equation*}
\bar{w}_p=\frac{1}{|\omega_{a_p}|}\sum_{K'\in\omega_{a_p}}w|_{K'}({a_p}),
 \end{equation*}
 as defined in \eqref{interpolationC}. This gives
\begin{equation*}
  \begin{split}
  \|w-\Pi_{\rm c}w\|^2&=\sum_{K\in\mathcal{T}}\|w-\Pi_{\rm c}w\|^2_{L^2(K)}\\
  &=\sum_{K\in\mathcal{T}}\|\sum^{n+1}_{p=1} w|_K(a_p)\theta_p-\sum^{n+1}_{p=1}\bar{w}_p\theta_p\|^2_{L^2(K)}\\
  &\leq\sum_{K\in\mathcal{T}}\sum^{n+1}_{p,q=1}\left|(w|_K(a_p)-\bar{w}_p)(w|_K(a_q)-\bar{w}_q)\right|(\theta_p,\theta_q)_{L^2(K)}.
  \end{split}
  \end{equation*}
  An explicit calculation that $(\theta_p,\theta_q)_{L^2(K)}=\frac{|K|}{(n+1)(n+2)}(1+\delta_{pq})$ leads to
   \begin{equation*}
 \|w-\Pi_{\rm c}w\|^2
  \leq\sum_{K\in\mathcal{T}}\frac{|K|}{n+1}\sum_{p=1}^{n+1}|w|_K(a_p)-\bar{w}_p|^2.
  \end{equation*}
It follows from the definitions of the interpolation operator $\Pi_{\rm c}$ in \eqref{interpolationC} and $N$ in \eqref{elementnumber} that
\begin{equation}\label{sec6eq1}
\begin{split}
  \|w-\Pi_{\rm c}w\|^2&\leq \sum_K\frac{|K|}{n+1}\sum_{p=1}^{n+1}\sup_{K'\cap{a_p}\neq\varnothing}|w|_K(a_p)-w|_{K'}(a_p)|^2\\
  &\leq \sum_{K\in\mathcal{T}}\frac{|K|}{n+1}\sum_{p=1}^{n+1}\frac{N}{4}\sum_{E'\in\mathcal{E},E'\cap{a_p}\neq\varnothing}|[w]|^2_{L^\infty(E')}\\
  &= \sum_{K\in\mathcal{T}}\frac{N|K|}{4(n+1)}\sum_{p=1}^{n+1}\sum_{E'\in\mathcal{E},E'\cap{a_p}\neq\varnothing}|[w]|^2_{L^\infty(E')}.
  \end{split}
  \end{equation}
Given $E'\in\mathcal{E}$, suppose that $|[w]|$ achieves the maximum at point $z'$ and the centroid of $E'$ is $M'$. Let $\tau_{E'}$ denote the tangent vector of $E'$ from $M'$ to $z'$. Since $\int_{E'}[w]ds=0$ and $[w]\in P_1(E')$, this yields that
\begin{equation}\label{sec6eq2}
\begin{split}
|[w](z')|=&\left|\int_M^z[\frac{\partial w}{\partial\tau_{E'}}]ds\right|\leq|z'-M'|\|[\nabla w]\|_{L^\infty(E')}\\
&\leq\frac{n-1}{n}h_{E'}\|[\nabla w]\|_{L^\infty(E')}=\frac{(n-1)h_{E'}}{n|E'|^{1/2}}\|[\nabla w]\|_{L^2(E')}.
\end{split}
\end{equation}
Substituting \eqref{sec6eq2} into \eqref{sec6eq1} gives that
  \begin{equation*}
  \begin{split}
 \|w-\Pi_{\rm c}w\|^2&\leq \sum_K\frac{(n-1)^2N|K|}{4n^2(n+1)}\sum^{n+1}_{p=1}\sum_{E'\in\mathcal{E},E'\cap{a_p}\neq\varnothing}h_{E'}^2 \|[\nabla w]\|_{L^2(E')}^2.
   \end{split}
\end{equation*}
Since $\nabla_{\rm NC}w$ is a piecewise constant, the trace inequality holds
\begin{equation*}
  \|[\nabla w]\|_{L^2(E')}^2\leq\frac{2|E'|}{|K_1|}\|\nabla w\|_{L^2(K_1)}^2+\frac{2|E'|}{|K_2|}\|\nabla w\|_{L^2(K_2)}^2.
\end{equation*}
Hence
  \begin{equation*}
  \begin{split}
 \|w-\Pi_{\rm c}w\|^2&\leq \sum_{K\in\mathcal{T}}\frac{N(n-1)^2|K|}{n^2(n+1)}\sum^{n+1}_{p=1}\sum_{K'\cap a_p\neq\varnothing}\frac{h^2_{E'}}{|K'|}\|\nabla w\|_{L^2(K')}^2.
  \end{split}
\end{equation*}
By the definition of $\xi$ in \eqref{rate}, there holds that
\begin{equation*}
\|w-\Pi_{\rm c}w\|^2\leq\frac{(n-1)^2N^2\xi}{n^2}h^2\sum_{K\in\mathcal{T}}\|\nabla w\|_{L^2(K)}^2.
\end{equation*}
This completes the proof.
\end{proof}
\section{Numerical Results}
\subsection{The Laplace operator}
In this example, the L-shape domain $\Omega=(0,1)^2/[0.5,1]^2$ and $A(x)\equiv1$. We compare the lower bounds provided by the CR and GCR elements.
Let $\lam_{\ell,\rm CR}$ be the $\ell$-th eigenvalues by the CR element. Carstensen et al. \cite{CarstensenGedickel2013} give the guaranteed lower bounds
\begin{equation}\label{CRlowerbound}
GLB_{\ell, \rm CR}=\frac{\lam_{\ell,\rm CR}}{1+0.1931\lam_{\ell,\rm CR}h^2}.
\end{equation}
By the GCR element, Theorem \ref{guranteelowerbound} gives the guaranteed lower bounds
\begin{equation}\label{GCRlowerbound}
 GLB_{\ell, \rm GCR}=\frac{\lambda_{\ell,\rm GCR}}{1+\frac{\lambda_{\ell,\rm GCR}^2h^4}{58.7276( 14.6819+\lambda_{\ell, \rm GCR}h^2)}}.
\end{equation}
Note that the lower bounds in \eqref{GCRlowerbound} have higher order accuracy than those in \eqref{CRlowerbound}. Table \ref{firsteigenvalue} and Table
\ref{20theigenvalue} show the results of first and 20th eigenvalues, respectively. For comparison, the discrete eigenvalues $\lam_{\ell,\rm P1}$ by the conforming P1 element are computed as upper bounds. Due to the fact that $V_{\rm CR}\subset V_{\rm GCR}$, $\lam_{\ell, \rm GCR}$ is smaller than $\lam_{\ell,\rm CR}$.  However, the guaranteed lower bounds produced by the GCR element are larger than those by the CR element.

\begin{table}[h!]
\centering
\caption{The first eigenvalue of L-shape domain}
\label{firsteigenvalue}
\begin{tabular}{|c|c|c|c|c|c|}
\hline
        $h$ & $\lam_{1,\rm CR}$ &     $GLB_{1, \rm CR}$ & $\lam_{1,\rm GCR}$ &     $ GLB_{1, \rm GCR}$ &$ \lam_{1,\rm P1}$ \\
\hline
  0.707107 &         24 &   11.6092 &    21.4979 &   19.9542 &            \\
\hline
  0.353553 &    32.7371 &   24.0013 &    31.1326 &    30.7063 &     56.3170 \\
\hline
  0.176777 &    36.5336 &   33.1658 &    35.9771 &   35.9282 &    43.0976 \\
\hline
  0.088388 &    37.8448 &   36.8751 &     37.6910 &   37.6873 &    39.8639 \\
\hline
  0.044194 &   38.2993 &   38.0462 &    38.2596 &   38.2594 &    38.9633 \\
\hline
  0.022097 &   38.4619 &    38.3978 &    38.4519 &   38.4519 &    38.6918 \\
\hline
  0.011049 &    38.5219 &   38.5058 &    38.5194 &   38.5194 &    38.6048 \\
\hline
  0.005524 &    38.5446 &   38.5406 &     38.5440     &   38.5440       &  38.5754
      \\
\hline
\end{tabular}
\end{table}
\begin{table}[h!]
\centering
\caption{The 20th eigenvalue of L-shape domain}
\label{20theigenvalue}
\begin{tabular}{|c|c|c|c|c|c|}
\hline
          $h$ & $\lam_{20,\rm CR}$ &    $GLB_{20, \rm CR}$& $\lam_{20,\rm GCR}$ &      $GLB_{20, \rm GCR}$&$ \lam_{20,\rm P1}$  \\
\hline
  0.353553 &   454.2769 &   75.0788 &    298.6560 &    205.0860 &            \\
\hline
  0.176777 &   307.4914 &   165.7926 &   280.6304 &   265.7885 &   722.3323 \\
\hline
  0.088388 &   387.1673 &   305.0883 &   372.4979 &   369.4693 &   500.4567 \\
\hline
  0.044194 &   401.4816 &   375.3058 &   397.2255 &   396.9623 &   429.3377 \\
\hline
  0.022097 &   405.0899 &   398.0864 &   403.9846 &   403.9666 &   412.1292 \\
\hline
  0.011049 &   406.0462 &    404.2640 &   405.7671 &   405.7659 &   407.8798 \\
\hline
  0.005524 &   406.3103 &   405.8627 &   406.2404 &   406.2403 &   406.8021 \\
\hline
\end{tabular}
\end{table}
\subsection{General second elliptic operators}
In this example, let $\Omega=(0,1)^2$, and
\begin{equation*}
A(x)=\begin{pmatrix}
x_1^2+1 &x_1x_2\\
x_1x_2 & x_2^2+1
\end{pmatrix}.
\end{equation*}
By a direct computation, the eigenvalues of $A(x)$ are $x_1^2+x_2^2+1$ and $1$, and $|A-\bar{A}|_\infty\leq\min\{\frac{4}{3}h,1\}$.
The constants in \eqref{constant1}--\eqref{constant4} are
\begin{equation*}
  C_A=1,C_{\bar{A}}=1, C_{\bar{A},A}=\min\{\sqrt{1+\frac{8}{3}h},\sqrt{3}\},C_\infty=\min\{\frac{8}{3},\frac{2}{h}\}.
\end{equation*}
\begin{equation*}
  \eta_1=C_{\bar{A}}C_{\bar{A},A}=\min\{\sqrt{1+\frac{8}{3}h},\sqrt{3}\},
\end{equation*}
\begin{equation*}
  \eta_2=C_\infty C_{\bar{A}}C_AC_{\bar{A},A}=\min\{\frac{8}{3},\frac{2}{h}\}\min\{\sqrt{1+\frac{8}{3}h},\sqrt{3}\}.
\end{equation*}
To compute the guaranteed lower and upper bounds for the first eigenvalue, it doesn't need the meshsize condition in \eqref{hconstrain} and \eqref{hconstrain2}. As for the 20th eigenvalue, we compute $\lam^m_{20,c}$ as a upper bound of $\lam_{20}$. Then \eqref{hconstrain} reads as follows
\begin{equation*}
  h<\frac{j_{1,1}}{\eta_1\sqrt{\lam^m_{20,c}}}:=h_1.
\end{equation*}
Since the computations are on uniform partitions, the constants in \eqref{rate} and \eqref{elementnumber} are
\begin{equation*}
  \xi=1,N=6,\beta_2=\frac{N\sqrt{\xi}}{2}=3.
\end{equation*}
We use the estimate of $\beta_1$ in \eqref{estimateofbeta1}. Let $\beta_1\approx0.2984$.
The condition in \eqref{hconstrain2} reads
\begin{equation*}
  h<\frac{1}{(\beta_1+\beta_2)C_A\sqrt{\lam_{20,\rm GCR}}}=\frac{1}{(0.2984+3)\sqrt{\lam_{20,\rm GCR}}}:=h_2.
\end{equation*}
Let $\beta=1/2$ in Theorem \ref{guranteelowerbound}. The GCR element gives the guaranteed lower bounds
\begin{equation}
 GLB_{\ell, \rm GCR}=\frac{\lambda_{\ell,\rm GCR}}{1+\frac{\lambda_{\ell, \rm GCR}^2C_A^4h^4}{58.7276( 7.3410+\lambda_{\ell, \rm GCR}C_A^2h^2)}+2\eta_2^2h^2}.
\end{equation}

Table \ref{firsteigenvalue2} and Table
\ref{20theigenvalue2} show the results of the first and 20th eigenvalues, respectively. From Table \ref{20theigenvalue2}, we find that when $h\leq0.0110$, the conditions $h<h_1$ and $h<h_2$ are guaranteed. Actually, when $h\leq0.1768$, $\Pi_{\rm c}(\Pi_{\rm CR}E_{20,\rm GCR})$ is already $20$-dimensional and $\lam^m_{20,\rm c}$ is thus a guaranteed upper bound of $\lam_{20}$.
\begin{table}[h!]
\centering
\caption{The first eigenvalue of square domain}
\label{firsteigenvalue2}
\begin{tabular}{|r|r|r|r|r|}
\hline
   $h$       &    $\lam_{1,\rm GCR}$ &      $GLB_{1, \rm GCR}$ &         $\lam_{1,\rm P1}$ &     $\lam_{1,c}$ \\
\hline
  1.4142 &  22.93710 &   0.89342  &            &            \\
\hline
  0.7071&  22.73488 &   1.05071  &         39 &         39 \\
\hline
  0.3536 &  25.38568  &  5.67888 &  30.22432  &  30.68603  \\
\hline
  0.1768 &  26.29812  &  15.88658 &  27.52878  &  27.63606  \\
\hline
  0.0884 &  26.54494  &  23.33831 &  26.85419  &  26.86946  \\
\hline
  0.0442 &  26.60805  &  25.80656  &  26.68551  &  26.68745  \\
\hline
  0.0221 &  26.62394  &  26.42958  &  26.64332  &  26.64356  \\
\hline
  0.0110 &  26.62792  &  26.58041  &  26.63277  &  26.63280  \\
\hline
  0.0055 &  26.62892  &  26.61719  &  26.63013  &  26.63013  \\
\hline
\end{tabular}

\end{table}
\begin{table}
\caption{The 20th eigenvalue of square domain}
\label{20theigenvalue2}
\begin{tabular}{|r|r|r|r|r|r|r|r|}
\hline
   $h$ &         $h_1$   &       $h_2$     &       $\lam_{20,\rm GCR}$      &   $GLB_{20, \rm GCR}$     &    $\lam_{20,\rm P1}$         &    $\lam_{20,\rm c}$    &   $\lam^m_{20,\rm c}$          \\
\hline
0.3536   &             & 0.0197   &  236.8297 & 48.7524
 &&348.5134&\\
\hline
   0.1768  &    0.0874  &    0.0173  &  305.4755  &   174.9729
&  576.1674  &  620.3720
   & 720.0317 \\
\hline
   0.0884  &   0.1127   &    0.0159  &  362.8685  & 315.3326 &  427.1357  & 424.3606
  &  433.1020\\
\hline
   0.0442  & 0.1181   &    0.0156  &  378.9545  & 367.1308
  &  394.1451  & 394.3686
 &  394.7023 \\
\hline
   0.0221  &    0.1193  &    0.0155  &  383.2543  &  380.4266  &  387.0340  & 387.0722   &  387.0910  \\
\hline
   0.0110  &    0.1195  &    0.0155  &  384.3485  &  383.6609  &  385.2930  &  385.2979 &   385.2991  \\
\hline
   0.0055  &    0.1196  &    0.0155  &  384.6233  &  384.4539  &  384.8595  &  384.8601 &384.8601     \\
\hline
\end{tabular}
\end{table}

\end{document}